\documentclass[a4paper,11pt]{article}

\usepackage[left=2cm,top=2cm,right=2cm,bottom=2cm,bindingoffset=0.5cm]{geometry}

\usepackage{amsfonts,graphicx,caption,subcaption,amsmath,amssymb,
enumitem,url,authblk,amsthm,multicol,cite,float,bigints,tabularx}
\usepackage[rgb,dvipsnames]{xcolor}

\DeclareMathOperator\erf{erf}
\DeclareMathOperator\erfc{erfc}
\DeclareMathOperator\erfi{erfi}

\theoremstyle{definition}
\newtheorem{definition}{Definition}[section]
\theoremstyle{theorem}
\newtheorem{theorem}{Theorem}[section]
\theoremstyle{lemma}
\newtheorem{lemma}{Lemma}[section]
\theoremstyle{corollary}

\theoremstyle{theorem}
\newtheorem{remark}{Remark}
\theoremstyle{theorem}

\title{Approximation of the modified error function}

\author[1,2]{Andrea N. Ceretani}
\author[3,4]{Natalia N. Salva}
\author[1]{Domingo A. Tarzia}

\affil[1]{{\footnotesize CONICET - Depto. de Matem\'atica, Facultad de Ciencias Empresariales, Univ. Austral, Paraguay 1950, S2000FZF Rosario, Argentina.}}
\affil[2]{{\footnotesize Depto. de Matem\'atica, Facultad de Ciencias Exactas, Ingenier\'ia y Agrimensura, Univ. Nacional de Rosario, Pellegrini 250, S2000BTP Rosario, Argentina.}}
\affil[3]{{\footnotesize CONICET - CNEA, Depto. de Mec\'anica Computacional, Centro At\'omico Bariloche, Av. Bustillo 9500, 8400 Bariloche, Argentina.}}
\affil[4]{{\footnotesize Depto. de Matem\'atica, Centro Regional Bariloche, Univ. Nacional del Comahue, Quintral 250, 8400 Bariloche, Argentina.}}

\date{}

\begin{document}  
\maketitle

\begin{abstract}
In this article, we obtain explicit approximations of the modified error function introduced in {\em Cho, Sunderland. Journal of Heat Transfer} 96-2 (1974), 214-217, as part of a Stefan problem with a temperature-dependent thermal conductivity. This function depends on a parameter $\delta$, which is related to the thermal conductivity in the original phase-change process. We propose a method to obtain approximations, which is based on the assumption that the modified error function admits a power series representation in $\delta$. Accurate approximations are obtained through functions involving error and exponential functions only. For the special case in which $\delta$ assumes small positive values, we show that the modified error function presents some characteristic features of the classical error function, such as monotony, concavity, and boundedness. Moreover, we prove that the modified error function converges to the classical one when $\delta$ goes to zero.
\end{abstract}

dependent thermal conductivity, nonlinear second order ordinary differential equation.


\section{Introduction}\label{Sect:Introduction}
Phase-change processes are present in a broad variety of natural, technological and industrial situations \cite{AkMoSt2014,BoIv2014,CaFuFa2016,ElKh2013,FuFaPr2014,
GaToTu2009,KuHa2016}. Modelling them properly is then crucial for understanding or predicting the evolution of many physical processes. One common assumption when modelling phase-change processes is to consider constant thermophysical properties. Nevertheless, it is known that certain 
materials present properties which seem to obey other laws. Recently, some models including variable latent heat, density, melting temperature or thermal conductivity have been proposed in \cite{BaMcHsMo2014,BoTa2018,Fo2018-Manuscript,FoMyMi2015,RiMy2016,SaTa2011-a,VoSwPa2004}.

In this sense, in 1974, Cho and Sunderland presented a similarity solution for a Stefan problem in which the thermal conductivity is a linear function of the temperature distribution \cite{ChSu1974}. It is well known that similarity solutions to Stefan problems with constant coefficients can be expressed in terms of the error function $\erf$,
\begin{equation}\label{erf}
\erf(x)=\frac{2}{\sqrt{\pi}}\displaystyle\int_0^x\exp(-\xi^2)d\xi\quad x>0.
\end{equation}
In contrast to this, the solution obtained by Cho and Sunderland involves another function, which they have called {\em modified error function}. It was defined as the solution to a nonlinear boundary value problem, and its \mbox{existence} and uniqueness was recently proved in \cite{CeSaTa2017} for thermal conductivities with moderate variations. In spite of the latter, the modified error function was widely used for solving diffusion problems \cite{BrNaTa2007,CoKa1994,FrVi1987,Lu1991,OlSu1987,SaTa2011-b,Ta1998}, even before it was formally introduced by Cho and Sunderland in 1974 \cite{Cr1956,Wa1950}. 

When phase-change processes come from technological or industrial problems, not only appropriate models are required but also their solutions (or, at least, some properties of them). Sometimes, when explicit solutions are not known, models are solved through numerical methods which are tested with experimental data. When the latter are not available, one common practice is to test numerical methods by applying them to another problem whose explicit solution is known. Thus, having explicit solutions to models for phase-change processes is sometimes quite useful. Many works have been done in this direction, see for example %\cite{DeEtAl1989,EvKi2000,Go2002,McWuHi2008,Ta2011,VoSwPa2004}.
\cite{BoTa2018-a,BrNa2016,BrNa2017,CeSaTa2018,CeTa2014,DeEtAl1989,EvKi2000,
FaGuPrRu1993,FaPrTa1999,Go2002,Ho1988,McWuHi2008,NaTa2000,
Ro1986,SoWiAl1983,Ta1981-1982,Ta2011,Ta2017,TrBr1994,VoSwPa2004,
WiSo1986}.
Regarding the model in \cite{ChSu1974}, explicit solutions are not known yet. Aiming to make a contribution in this sense, the main goal of this article is to propose some approximations of the modified error function.

In order to present our ideas clearly, we briefly recall how the modified error function arises from the original phase-change process. For simplicity, we consider the case of a one-phase melting problem for a semi-infinite slab with phase-change temperature $T_m$, whose boundary $x=0$ is maintained at a constant temperature $T_\infty>T_m$. For this case, the thermal conductivity from Cho and Sunderland is 
\begin{equation}\label{k}
k(T)=k_0\left\{1+\delta\left(\frac{T-T_\infty}{T_m-T_\infty}\right)\right\},
\end{equation}
where $k_0>0$ is the thermal conductivity at $x=0$, and $\delta$ is some dimensionless parameter. Since $T=T_m$ at the free boundary, $\delta>-1$ becomes a necessary condition to assure the thermal conductivity is positive when $x=s(t)$. When the temperature distribution is assumed to be in the form $T(x,t)=A+B\Phi_\delta\left(\frac{x}{2\sqrt{\alpha_0t}}\right)$,\footnote{$\alpha_0$ is the coefficient of diffusion at $t=0$.} for $A$ and $B$ constant, one obtains that $\Phi_\delta$ can be found by solving the \mbox{following} nonlinear boundary value problem (see details in \cite{ChSu1974}): 
\begin{subequations}\label{Pb:y}
\begin{align}
\label{eq:y}&[(1+\delta y(x))y'(x)]'+2xy'(x)=0\quad 0<x<+\infty\\
\label{cond:0}&y(0)=0\\
\label{cond:infty}&y(+\infty)=1.
\end{align}
\end{subequations}  
The solution $\Phi_\delta$ to this problem is the already mentioned {\em modified error function}. Some plots for $\Phi_\delta$ are shown in Figure \ref{Fig:FEM}. They were obtained by numerically solving problem (\ref{Pb:y}) for $\delta=-0.9,-0.5,0,0.5,1,2$. For the special case in which $\delta=0$, which corresponds to a constant thermal conductivity, one finds that the modified error function coincides with the classical one. The coincidence is stronger than that shown from the numerical computations, since it can be easily proved that the error function is the only solution to problem (\ref{Pb:y}) when $\delta=0$. As $\delta$ moves away from zero, the modified error function differs more and more from the classical one. Nevertheless, both functions seem to share some properties (such as non-negativity, boundedness and rapid convergence to 1 when $x\to+\infty$). Moreover, when $\delta>0$, the modified error function seems to be increasing and concave, as the error function is. Observe that $-1<\delta<0$ is related to thermal conductivities that decrease when temperature increases (e.g. lead, methanol), whereas $\delta>0$ corresponds to thermal conductivities that increase as the temperature does (e.g. glycerin, mercury).
\begin{figure}[H] 
\caption{{\em Modified error function $\Phi_\delta$ for $delta=-0.9,-0.5,0,0.5,1,2$ over different domains.}}
\label{Fig:FEM}
\centering
\hspace*{-2cm}
\begin{subfigure}{0.4\textwidth}
\includegraphics[scale=0.35]{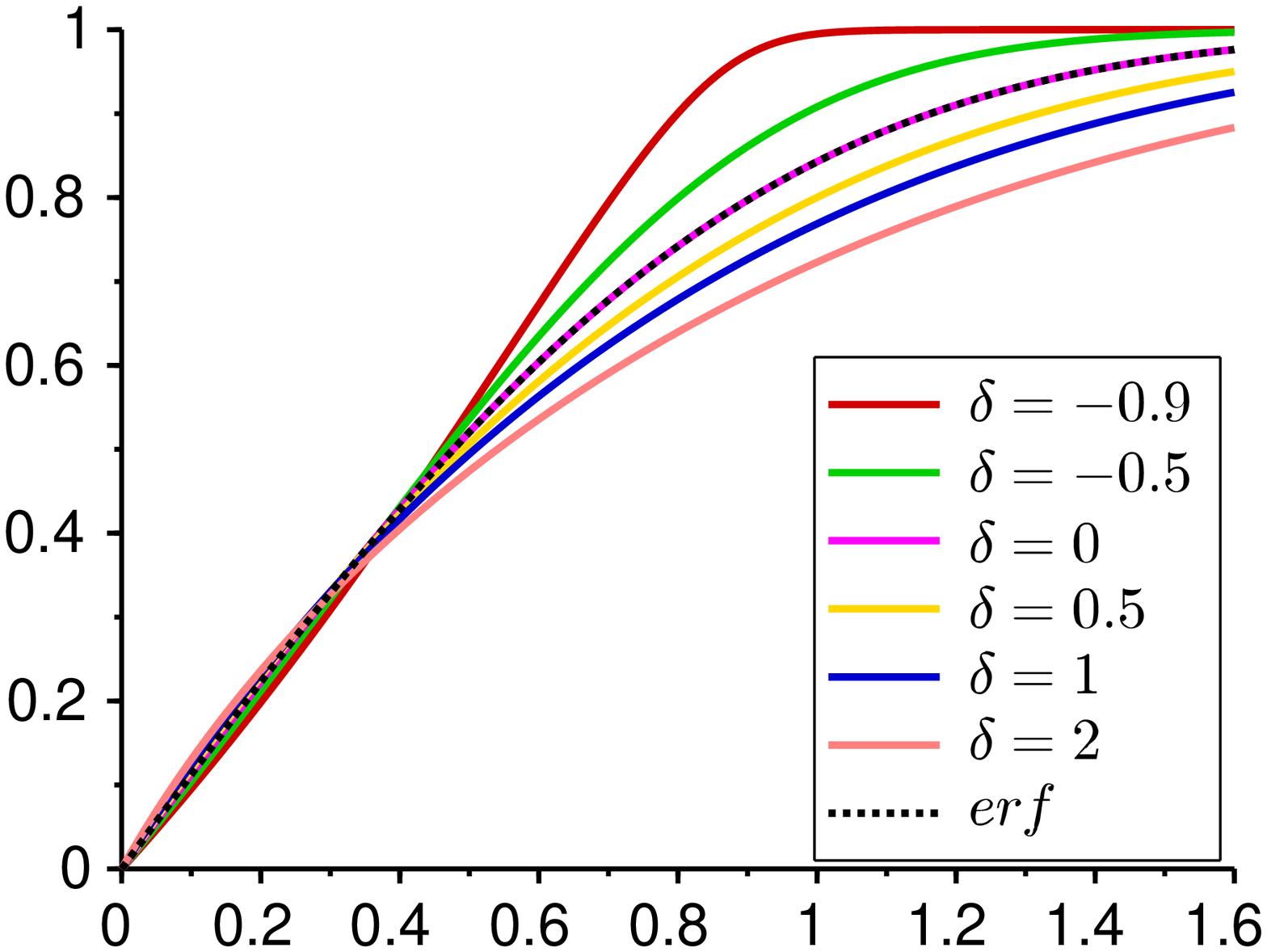}
\caption{$x\in[0,1.6]$}
\end{subfigure}
\hspace*{1.5cm}
\begin{subfigure}{0.4\textwidth}
\includegraphics[scale=0.35]{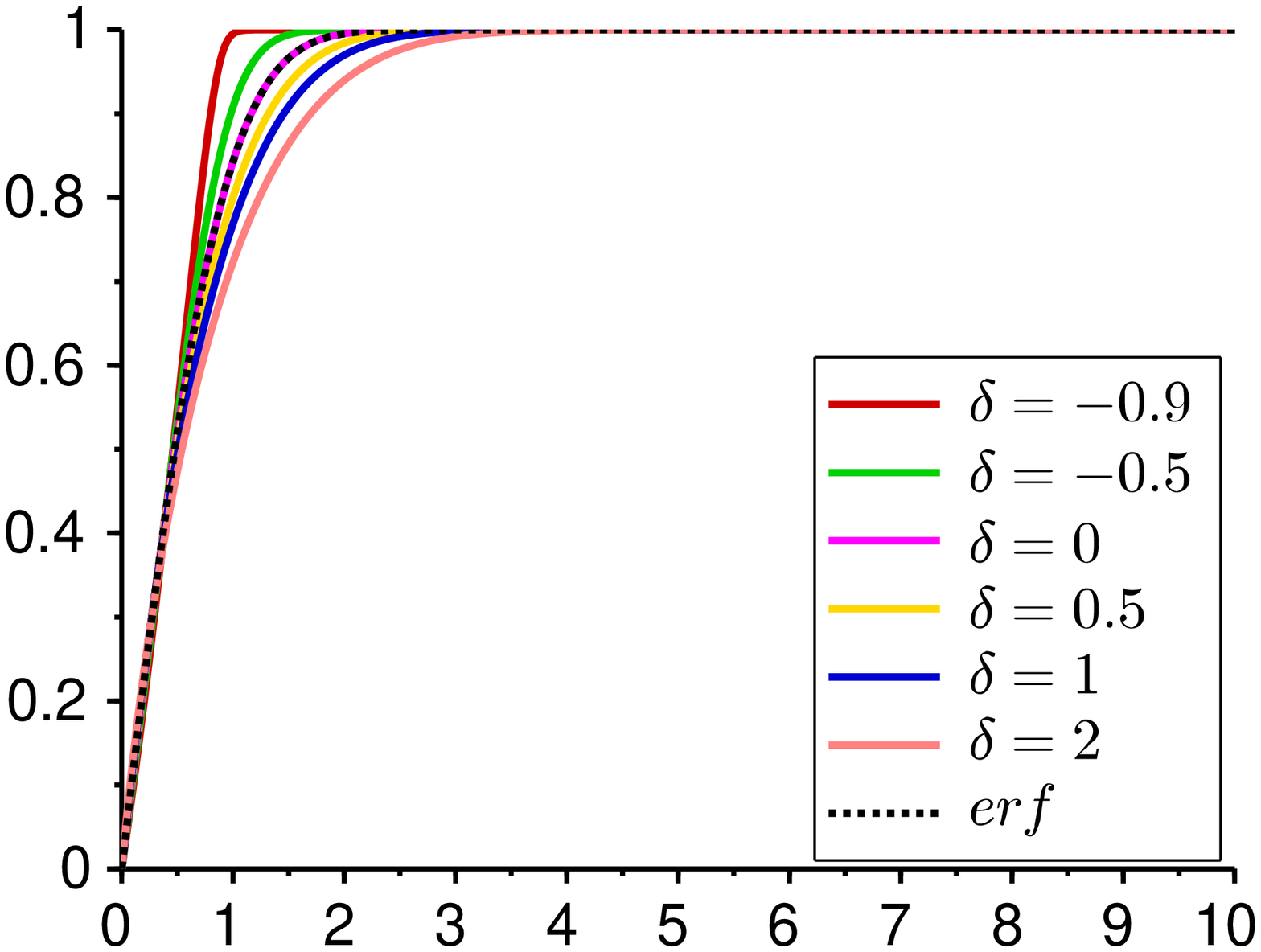}
\caption{$x\in[0,10]$}
\end{subfigure}
\end{figure}

\noindent Finally, we recall that the existence and uniqueness of 
$\Phi_\delta$ in the set of non-negative bounded analytic functions was recently proved in \cite{CeSaTa2017} for small positive values of $\delta$ (i.e. for increasing thermal conductivities that present moderate variations with respect to their initial value). Moreover, an upper bound $\delta_0$ for the parameter $\delta$ was characterized as the unique positive solution to the equation: 
\begin{equation}\label{eq:delta0}
\frac{x}{2}(1+x)^{3/2}(3+x)[1+(1+x)^{3/2}]=1\quad x>0.
\end{equation}

The approximations for the modified error function proposed in this article are based on the assumption that $\Phi_\delta$ admits a power series representation in the parameter $\delta$. More precisely, we assume that there exist functions $\varphi_n$ defined on $\mathbb{R}^+$ such that:
\begin{equation}\label{Phi}
\Phi_\delta(x)=\displaystyle\sum_{n=0}^\infty\varphi_n(x)\delta^n\quad x>0,
\end{equation}
and look for approximations $\Psi_{\delta,m}$ of the form
\begin{equation}\label{Psi}
\Psi_{\delta,m} (x)=\displaystyle\sum_{n=0}^m \varphi_n(x)\delta^n\quad x>0\quad\text{for }m\in\mathbb{N}_0.
\end{equation}

The organization of the article is as follows. First (Section \ref{Sect:Series}), we formally characterize each function $\varphi_n$ as the solution to a linear boundary value problem for a second order differential equation. The latter is homogeneous when $n=0$, but presents a non-zero source term dependent on $\varphi_k$ for $k=0,\dots, n-1$, when $n\in\mathbb{N}$. Then (Section \ref{Sect:Zero}), we present the zero order approximation $\Psi_{\delta,0}$. We find that it is given by the classical error function, and we prove that $\Phi_\delta$ uniformly converges to $\Psi_{\delta,0}=\erf$ when $\delta$ goes to zero. Since theoretical results are presented, the proof will be given only for those values of $\delta$ for which existence and uniqueness of the modified error function is known (i.e. we only consider $\delta\to 0^+$). After that (Section \ref{Sect:OneTwo}), we present the first and second order approximations $\Psi_{\delta,1}$ and $\Psi_{\delta,2}$. We obtain that $\Psi_{\delta,1}$ can be explicitly written in terms of the error and exponential functions only. By contrast, $\Psi_{\delta,2}$ involves some integrals whose values (explicit dependence on $x$) are not yet available in the literature. We analyze numerical errors between the approximations $\Psi_{\delta,1}$, $\Psi_{\delta,2}$, and the modified error function $\Phi_{\delta}$ when $\delta$ assumes small positive values, thus the  existence and uniqueness of the modified error is assured. From them, we conclude that the first order approximation is better than the approximation of order two. Moreover, we find that it is also better than the approximation of zero order. We present some plots comparing $\Psi_{\delta,1}$ and $\Phi_\delta$ for some values of $\delta>-1$, which show good agreement. Finally (Section \ref{Sect:Prop}) we restrict the arguments again to small positive values of $\delta$ and prove that the modified and classical error functions share the properties of being increasing, concave and bounded functions.  
 
%%==================================================
\section{Formal series representation of the modified error function}\label{Sect:Series}
%%==================================================
This Section is devoted to obtain a formal characterization of the coefficients $\varphi_n$ in the power series representation of the modified error function $\Phi_\delta$ given by (\ref{Phi}). 

Let $\delta>-1$ and $x>0$ be given. When $\Phi_\delta$ is defined by (\ref{Phi}), formal computations from equation (\ref{eq:y}) yield
\begin{equation}\label{EqPhiDelta-1}
\begin{split}
&\displaystyle\sum_{n=0}^\infty\displaystyle\sum_{m=0}^\infty
\delta^{n+m+1}\left(\varphi'_n(x)\varphi'_m(x)+
\varphi_n(x)\varphi_m''(x)\right)\\
&+
\displaystyle\sum_{n=0}^\infty\delta^n\left(\varphi''_n(x)+2x\varphi'_n(x)\right)
=0\quad x>0.
\end{split}
\end{equation} 

By introducing the following notation:
\begin{subequations}
\begin{align}
&a(x,n,m)=\varphi_n'(x)\varphi_m'(x)+\varphi_n(x)\varphi_m''(x)&x>0,\,n,m\in\mathbb{N}_0\\
&b(x,n)=\varphi_n''(x)+2x\varphi'_n(x)&x>0,\,n\in\mathbb{N}_0,
\end{align}
\end{subequations}
equation (\ref{EqPhiDelta-1}) can be written as

\begin{equation}
\displaystyle\sum_{n=1}^\infty\left(\displaystyle\sum_{k=1}^n
a(x,k-1,n-k)+b(x,n)\right)\delta^n+b(x,0)=0,\quad x>0.
\end{equation}
Therefore, the function $\Phi_\delta$ defined by (\ref{Phi}) is a formal solution to problem (\ref{Pb:y}) if and only if the functions $\varphi_n$, $n\in\mathbb{N}_0$, are such that
\begin{subequations}\label{PbPhinTodas}
\begin{align}
&\displaystyle\sum_{k=1}^na(x,k-1,n-k)+b(x,n)=0\quad x>0, \forall\,n\in\mathbb{N},
\quad b(x,0)=0, \, x>0,\\
&\varphi_n(0^+)=0\quad \forall\,n\in\mathbb{N}_0,
\hspace{1.5cm} \varphi_0(+\infty)=1,
\hspace{1cm} \varphi_n(+\infty)\quad \forall\,n\in\mathbb{N}.
\end{align}
\end{subequations}
That is, if and only if $\varphi_0$ and $\varphi_n$, $n\in\mathbb{N}$, are solutions to
\begin{subequations}\label{PbPhi0}
\begin{align}
&2x\varphi'_0(x)+\varphi_0''(x)=0&x>0\\
&\varphi_0(0^+)=0&\\
&\varphi_0(+\infty)=1&
\end{align}
\end{subequations}
and 
\begin{subequations}\label{PbPhiN}
\begin{align}
&2x\varphi'_n(x)+\varphi_n''(x)=-\displaystyle\sum_{k=1}^n\left(
\varphi'_{k-1}(x)\varphi_{n-k}'(x)+\varphi_{k-1}(x)\varphi_{n-k}''(x)\right)&x>0\\
&\varphi_n(0^+)=0&\\
&\varphi_n(+\infty)=0,&
\end{align}
\end{subequations}
respectively.
\begin{remark}
We find from (\ref{PbPhiN}a) that functions $\varphi_k$ must be known for $k=0,\dots, n-1$ in order to find $\varphi_n$.
\end{remark}

The next two Sections are dedicated to present and analyse the approximations $\Psi_{\delta,m}$ when $m=0,1,2$ and their coefficients are solutions to problems (\ref{PbPhi0}), (\ref{PbPhiN}) (see (\ref{Psi})). Each function $\Psi_{\delta,m}$ will be referred to as {\em approximation of order $m$}.

%%==================================================
\section{Approximation of order zero}\label{Sect:Zero}
%%==================================================

The approximation of order zero is $\Psi_{\delta,0}=\varphi_0$, where $\varphi_0$ is a solution of problem (\ref{PbPhi0}). We note that the latter coincides with (\ref{Pb:y}) when $\delta=0$. Thus, as it was already mentioned, its unique solution is the error function. Hence
\begin{equation}
\Psi_{\delta,0}(x)=\erf(x)\quad\quad x>0.
\end{equation} 

The remaining part of this Section is devoted to prove that the modified error function uniformly converges to the classical one, when the parameter $\delta$ goes to zero. We will restrict our analysis for those values of $\delta$ for which existence and uniqueness of the modified error function is known, i.e. to small positive values of $\delta$ \cite{CeSaTa2017}. Thus, our main goal will be to prove that
\begin{equation}
\epsilon_{\delta,0}\to 0\quad\quad\text{and}\quad\quad
\delta\to 0^+,
\end{equation}
where $\epsilon_{\delta,0}$ is the error between the classical and modified error functions, defined by 
\begin{equation}\label{Error0}
\epsilon_{\delta,0}=||\Phi_\delta-\erf||_\infty.
\end{equation}

The fact that the error function satisfies problem (\ref{Pb:y}) when $\delta=0$, suggests to analyse the dependence of problem (\ref{Pb:y}) on the parameter $\delta$. We begin by recalling the main result in \cite{CeSaTa2017}:

\begin{theorem}\label{Th}
Let $\delta_0$ be the only solution to equation (\ref{eq:delta0}), and let $0\leq \delta<\delta_0$ be given. Then there exist a unique solution $\Phi_\delta$ to problem (\ref{Pb:y}) in the set $K$ of all non-negative analytic functions in $\mathbb{R}_0^+$ which are bounded by 1. Moreover, $\Phi_\delta$ is given as the unique fixed point of the operator $\tau_\delta$ from $K$ to $K$ defined by
\begin{equation}\label{tau}
\tau_\delta (h)(x)=C_{\delta,h} \displaystyle\int_0^x \frac{1}{1+\delta h(\eta)}\exp\left(-2\displaystyle\int_0^\eta\frac{\xi}{1+\delta h(\xi)}d\xi\right)d\eta\quad x\geq 0,
\end{equation}
for $h\in K$, with $C_{\delta,h}$ given by
\begin{equation}\label{Ch}
C_{\delta,h}=\left(\displaystyle\int_0^{+\infty} \frac{1}{1+\delta h(\eta)}\exp\left(-2\displaystyle\int_0^\eta\frac{\xi}{1+\delta h(\xi)}d\xi\right)d\eta\right)^{-1}.
\end{equation}
\end{theorem} 

\begin{proof}
See \cite{CeSaTa2017}.
\end{proof}

\begin{remark}
\begin{enumerate}
\item[]
\item In \cite{CeSaTa2017}, the case $\delta=0$ was not considered. This is because it corresponds to the classical case, in which the thermal conductivity is constant and the modified error function is the classical one. Nevertheless, besides this physical consideration, all theorems in \cite{CeSaTa2017} are still valid when $\delta=0$.
\item From \eqref{k}, we see that $dk(T)/dT=\delta k_0/(T_m-T_\infty)$. Thus, the bound on $\delta$ established by Theorem \ref{Th} to prove the existence and uniqueness of the modified error function $\Phi_\delta$ determines a necessary condition on the data of the associated Stefan problem to obtain similarity solutions. This condition establishes that the velocity of change of the thermal conductivity with respect to changes on the temperature distribution must be controlled by some multiple of the ratio between the thermal conductivity $k_0$ at $x=0$, and the difference between the phase-change and boundary temperatures, $T_m$ and $T_\infty$. In other words, that $s<\delta_0 k_0/(T_m-T_\infty)$, where $s$ is the slope in the linear dependence of $k$ on $T$.
\end{enumerate}
\end{remark}

\begin{definition}\label{Def:Lips}
We say that problem (\ref{Pb:y}) is Lipschitz continuous on the parameter $\delta$ if 
\begin{equation}
\exists\,L>0\quad/\quad\forall\,\delta_1, \delta_2 \in [0,\delta_0)
\quad:\quad
||\Phi_{\delta_1}-\Phi_{\delta_2}||_\infty\leq L|\delta_1-\delta_2|,
\end{equation}
where $\Phi_{\delta_1}$, $\Phi_{\delta_2}$ are the only solutions in $K$ to problem (\ref{Pb:y}) with parameters $\delta_1$, $\delta_2$, respectively.   
\end{definition}

Thus, if problem (\ref{Pb:y}) is Lipschitz continuous on $\delta$, we find that the modified error function $\Phi_\delta$ converges uniformly on $x>0$ to the classical error function $\erf$, when $\delta\to 0^+$. In other words, that $\epsilon_{\delta,0}\to 0$ when $\delta\to 0^+$. Before proving the Lipschitz dependence of problem (\ref{Pb:y}) on $\delta$, we introduce some preparatory results in the following: 

%------------------------------------------------------------------
%
%                LEMMA
% 
%------------------------------------------------------------------

\begin{lemma}\label{Le:cotas}
Let $\delta_1, \delta_2 \in [0,\delta_0)$, $h, h_1, h_2\in K$ and $0\leq x\leq +\infty$ be given. The following estimations hold:
\begin{enumerate}
\item[]
\item[a)] \small $\displaystyle\bigintss_0^x\left|
\frac{\exp\left(-2\displaystyle\int_0^\eta\frac{\xi}{1+\delta_1 h(\xi)}d\xi\right)}{1+\delta_1 h(\eta)}-
\frac{\exp\left(-2\displaystyle\int_0^\eta\frac{\xi}{1+\delta_2 
h(\xi)}d\xi\right)}{1+\delta_2 h(\eta)}\right|d\eta\leq$ \linebreak 
$\displaystyle \dfrac{\sqrt{\pi}}{4} (1+\delta_0)^{1/2}(3+\delta_0)| \delta_1 
-\delta_2 |$ 
\normalsize
\item[]
\item[b)] $\displaystyle \left| \frac{1}{C_{h_1,\delta_1}}-\frac{1}{C_{h_2,\delta_2}} \right| \leq 
\frac{\sqrt{\pi}}{4} (1+\delta_0)^{1/2}(3+\delta_0)\left(\delta_0||h_1-h_2||_{\infty}+ |\delta_1 -\delta_2 |\right).$
\end{enumerate}
\end{lemma}

\begin{proof}
\begin{enumerate}
\item[]
 \item[a)] 
Let $f$ be the real function defined on $\mathbb{R}^+_0$ by $f(x)=\exp(-2x)$, and
\begin{equation*}
x_1= \int_0^\eta\frac{\xi}{1+\delta_1 h(\xi)}d\xi, \quad x_2= \int_0^\eta\frac{\xi}{1+\delta_2 h(\xi)}d\xi\quad (\eta>0 \text{ fixed}).
\end{equation*}

It follows from the Mean Value Theorem applied to function $f$ that
\begin{equation*}
\left| f(x_1)-f(x_2)\right|=|f'(u)| |x_1-x_2|,
\end{equation*}
where $u$ is a real number between $x_1$ and $x_2$. Without any lost of generality, we assume that $\delta_1\geq \delta_2$. Then, $x_1\leq x_2$ and we find
\begin{align*}
&|f'(u)|\leq |f'(x_1)|\leq  2 \exp \left( -\frac{\eta^2}{1+\delta_0}\right)\text{ since }||h||_\infty\leq 1,\\
& |x_1-x_2| \leq \frac{\eta^2}{2} |\delta_1 - \delta_2|.
\end{align*} 

Therefore, 
\begin{equation*}
\left|f(x_1)-f(x_2)\right| \leq  |\delta_1 - \delta_2|\eta^2  \exp \left( -\frac{\eta^2}{1+\delta_0}\right).
\end{equation*}

Then, we find
\begin{equation*}
\begin{split}
\left|
\frac{f(x_1)}{1+\delta_1 h(\eta)}-\frac{f(x_2)}{1+\delta_2 h(\eta)}
\right|&=
\left|
\frac{f(x_1)-f(x_2)}{1+\delta_1 h(\eta)}+\frac{f(x_2)h(\eta)(\delta_2-\delta_1)}{(1+\delta_1h(\eta))(1+\delta_2h(\eta))}\right|\\
&\leq |f(x_1)-f(x_2)|+|f(x_2)||\delta_1-\delta_2|\\
&\leq |\delta_1 - \delta_2|\exp\left(-\frac{\eta^2}{1+\delta_0}\right)(\eta^2+1).
\end{split}
\end{equation*}
The final bound is now obtained by integrating the last expression.

\item[b)] First, we find the estimation
\begin{equation}\label{invCdem}
\begin{split}
&\left| \frac{1}{C_{\delta_1,h_1}}-\frac{1}{C_{\delta_2,h_2}} \right| \leq
\left| \frac{1}{C_{\delta_1,h_1}}-\frac{1}{C_{\delta_1,h_2}} \right|\\
&\\
&+ \small \left|\bigintss_0^{+\infty}
\frac{\exp\left(-2\displaystyle\int_0^\eta\frac{\xi}{1+\delta_1 h_2(\xi)}d\xi\right)}{1+\delta_1 h_2(\eta)}-
\frac{\exp\left(-2\displaystyle\int_0^\eta\frac{\xi}{1+\delta_2 h_2(\xi)}d\xi\right)}{1+\delta_2 h_2(\eta)}d\eta\right|.
\end{split}
\end{equation}
Taking into consideration that the first term in the right hand side of (\ref{invCdem}) is bounded by $\frac{\sqrt{\pi}}{4}\delta_0 (1+\delta_0)^{1/2}(3+\delta_0)||h_1-h_2||_{\infty}$ (see \cite[Lemma 2.1]{CeSaTa2017}), the desired bound follows from (\ref{invCdem}) and item a). 
\end{enumerate} 
\end{proof}

%------------------------------------------------------------------
%
%                THEOREM
% 
%------------------------------------------------------------------
\begin{theorem}\label{theoLip}
Problem (\ref{Pb:y}) is Lipschitz continuous on the parameter $\delta$.
\end{theorem}

\begin{proof}
Let $\delta_1, \delta_2\in [0,\delta_0)$ be given, and let $\Phi_{\delta_1}, \Phi_{\delta_2}\in K$ be the solutions to problem (\ref{Pb:y}) with parameters $\delta_1$, $\delta_2$, respectively. 

Exploiting the fact that $\Phi_{\delta_i}$ is the fixed point of the operator $\tau_{\delta_i}$ defined by (\ref{tau}), $i=1,2$, we find
\begin{equation}\label{CotaLip}
\begin{split}
|\Phi_{\delta_1}(x)- \Phi_{\delta_2}(x)| &\leq
\left|  C_{\delta_1,\Phi_{\delta_1}}- C_{\delta_2,\Phi_{\delta_2}} \right| \int_0^\infty H(\Phi_{\delta_2},\delta_2)(\eta) d \eta\\
&+C_{\delta_1,\Phi_{\delta_1}} 
\left| \int_0^x H(\Phi_{\delta_1},\delta_1)(\eta)-H(\Phi_{\delta_2},\delta_1)(\eta) d \eta \right|\\
&+C_{\delta_1,\Phi_{\delta_1}} 
\left| \int_0^x H(\Phi_{\delta_2},\delta_1)(\eta)-H(\Phi_{\delta_2},\delta_2)(\eta) d \eta \right|\quad\forall\,x>0,
\end{split}
\end{equation}
where we have written
\begin{equation*}
H(h,\delta)(x)=\frac{\exp\left(-2\displaystyle\int_0^x\frac{\xi}{1+\delta h(\xi)}d\xi\right)}{1+\delta h(x)}\quad x>0,
\end{equation*}
with $h=\Phi_{\delta_i},\,\delta=\delta_j,\,i,j=1,2.$

From the estimations
\begin{enumerate}
\item[i)] $C_{\delta_1,\Phi_{\delta_1}}
\leq \frac{2\left(1+\delta_0\right)}{\sqrt{\pi}}$
\item[ii)] $\left|C_{\delta_1,\Phi_{\delta_1}}-C_{\delta_2,\Phi_{\delta_2}}\right|=
\left|\frac{1}{C_{\delta_1,\Phi_{\delta_1}}}-\frac{1}{C_{\delta_2,\Phi_{\delta_2}}}\right|C_{\delta_1,\Phi_{\delta_1}}C_{\delta_2,\Phi_{\delta_2}}$
\item[]$\leq 
\frac{1}{\sqrt{\pi}}(1+\delta_0)^{5/2}(3+\delta_0)\left(\delta_0||\Phi_{\delta_1}-\Phi_{\delta_2}||_\infty+|\delta_1-\delta_2|\right)$\\
(see Lemma \ref{Le:cotas})
\item[iii)] $\left| \int_0^x 
H(\Phi_{\delta_1},\delta_1)(\eta)-H(\Phi_{\delta_2},\delta_1)(\eta) d \eta 
\right|\leq$ 
\item[] $\displaystyle 
\frac{\sqrt{\pi}}{4}\delta_0(1+\delta_0)^{1/2}(3+\delta_0)||\Phi_{ \delta_1} 
-\Phi_{\delta_2}||_\infty$\\
(see \cite[Lemma 2.1]{CeSaTa2017})
\item[iv)] $\left| \int_0^x H(\Phi_{\delta_2},\delta_1)(\eta)-H(\Phi_{\delta_2},\delta_2)(\eta)  d \eta \right|\leq
\frac{\sqrt{\pi}}{4} (1+\delta_0)^{1/2}(3+\delta_0)| \delta_1 -\delta_2 |$\\ (see Lemma \ref{Le:cotas}),
\end{enumerate}
and (\ref{CotaLip}), we obtain
\begin{equation} \label{CotaLip2}
\left|\Phi_{\delta_1}(x)-\Phi_{\delta_2}(x)\right|
\leq C||\Phi_{\delta_1}-\Phi_{\delta_2}||_\infty+\frac{C}{\delta_0}|\delta_1-\delta_2|\quad\forall\,x>0,
\end{equation}
where $C=\delta_0(1+\delta_0)^{3/2}(3+\delta_0)$. 

Since $\delta_0$ is the solution to equation (\ref{eq:delta0}), we find
\begin{equation*}
C=\frac{2}{1+(1+\delta_0)^{3/2}}.
\end{equation*}
Thus, $0<C<1$. From this and (\ref{CotaLip2}), we obtain
\begin{equation*}
||\Phi_{\delta_1}-\Phi_{\delta_2}||_\infty\leq L|\delta_1-\delta_2|\quad\text{with }L=\frac{C}{\delta_0(1-C)}>0.
\end{equation*}
\end{proof}

In conclusion, we have found that the approximation of zero order $\Psi_{\delta,0}$ is the classical error function. Furthermore, for those values of $\delta$ for which existence and uniqueness of the modified error function $\Phi_\delta$ is known, we found that $\Phi_\delta$ uniformly converges to $\Psi_{\delta,0}=\erf$ on $x>0$. In particular, the latter suggest that $\Phi_\delta\simeq\Psi_{\delta,0}=\erf$ for small positive values of $\delta$, in agreement to Figure \ref{Fig:FEM}.

%%==================================================
\section{Approximations of order one and two}\label{Sect:OneTwo}
%%==================================================

The approximations of order one and two are given by
\begin{equation*}
\Psi_{\delta,1}=\varphi_0+\varphi_1\delta
\quad\quad\text{and}\quad\quad \Psi_{\delta,1}=\varphi_0+\varphi_1\delta+\varphi_2\delta^2,
\end{equation*}
respectively, where $\varphi_0$ is the solution to problem (\ref{PbPhi0}) and $\varphi_1$, $\varphi_2$ are the solutions to problem (\ref{PbPhiN}) for $n=1$ and $n=2$. From Section \ref{Sect:Zero}, we know that $\varphi_0=\erf$. It enables us to define the source term in the differential equation of problem (\ref{PbPhiN}) for $n=1$. This problem can be explicitly solved and, from its solution, it can be also defined and solved problem (\ref{PbPhiN}) for $n=2$. We summarize these results in the following

\begin{theorem}
\begin{enumerate}
\item[]
\item [a)] The only solution $\varphi_1$ to problem (\ref{PbPhiN}) for $n=1$ is given by
\begin{equation}\label{Phi1}
\begin{split}
\varphi_1(x)&=\left(\frac{1}{2}-\frac{1}{\pi}\right)\erf(x)+\frac{1}{\pi}\left\{1-\exp(-2x^2)\right\}\\
&-\frac{1}{\sqrt{\pi}}x\erf(x)\exp(-x^2)-\frac{1}{2}\erf^2(x)\quad x>0.
\end{split}
\end{equation}
\item [b)] The only solution $\varphi_2$ to problem (\ref{PbPhiN}) for $n=2$ is given by
\begin{equation}\label{Phi2}
\varphi_2(x)=\frac{\sqrt{\pi}}{2}g_2(x)\left[\displaystyle\int_0^x\erfc(y)\exp(y^2)dy-
\frac{\sqrt{\pi}}{2}\erfc(x)\erfi(x)\right]\quad x>0,
\end{equation}
where $g_2$, $\erfc$, $\erfi$ are the real functions defined in $\mathbb{R}^+$ by
\begin{subequations}
\begin{flalign*}
g_2(x)=&\frac{16}{\pi}\erf(x)\exp(-2x^2)+
\frac{4}{\pi}\left(\frac{2}{\pi}-1\right)\exp(-2x^2)\\
&-\frac{12}{\pi\sqrt{\pi}}x\exp(-3x^2)
+\left(\frac{4}{\sqrt{\pi}}-\frac{8}{\pi\sqrt{\pi}}\right)x\erf(x)\exp(-x^2)\\
&-\frac{12}{\sqrt{\pi}}x\erf^2(x)\exp(-x^2)+\frac{4}{\pi\sqrt{\pi}}x\exp(-x^2)\\
&-\frac{8}{\pi}x^2\erf(x)\exp(-2x^2)+
\frac{4}{\sqrt{\pi}}x^3\erf^2(x)\exp(-x^2)\\
\erfc(x)&=1-\erf(x)\\
\erfi(x)&=-i\erf(ix)=\frac{2}{\sqrt{\pi}}\displaystyle\int_0^x\exp(\xi^2)d\xi
\quad\text{($i$: imaginary unit)}.
\end{flalign*}
\end{subequations}
\end{enumerate}
\end{theorem}

\begin{proof}
It follows from standard results in ordinary differential equations (see, e.g. \cite{PoZa1995}).
\end{proof}

Plots for $\varphi_n$, $n=0,1,2$, are shown in Figure \ref{Fig:Coeff}.

\begin{figure}[H] 
\caption{{\em First three coefficients of the power series representation of 
the modified error function $\Phi_\delta$ (see (\ref{Phi})).}}
\label{Fig:Coeff}
\centering
\begin{subfigure}{0.3\textwidth}
\includegraphics[scale=0.25]{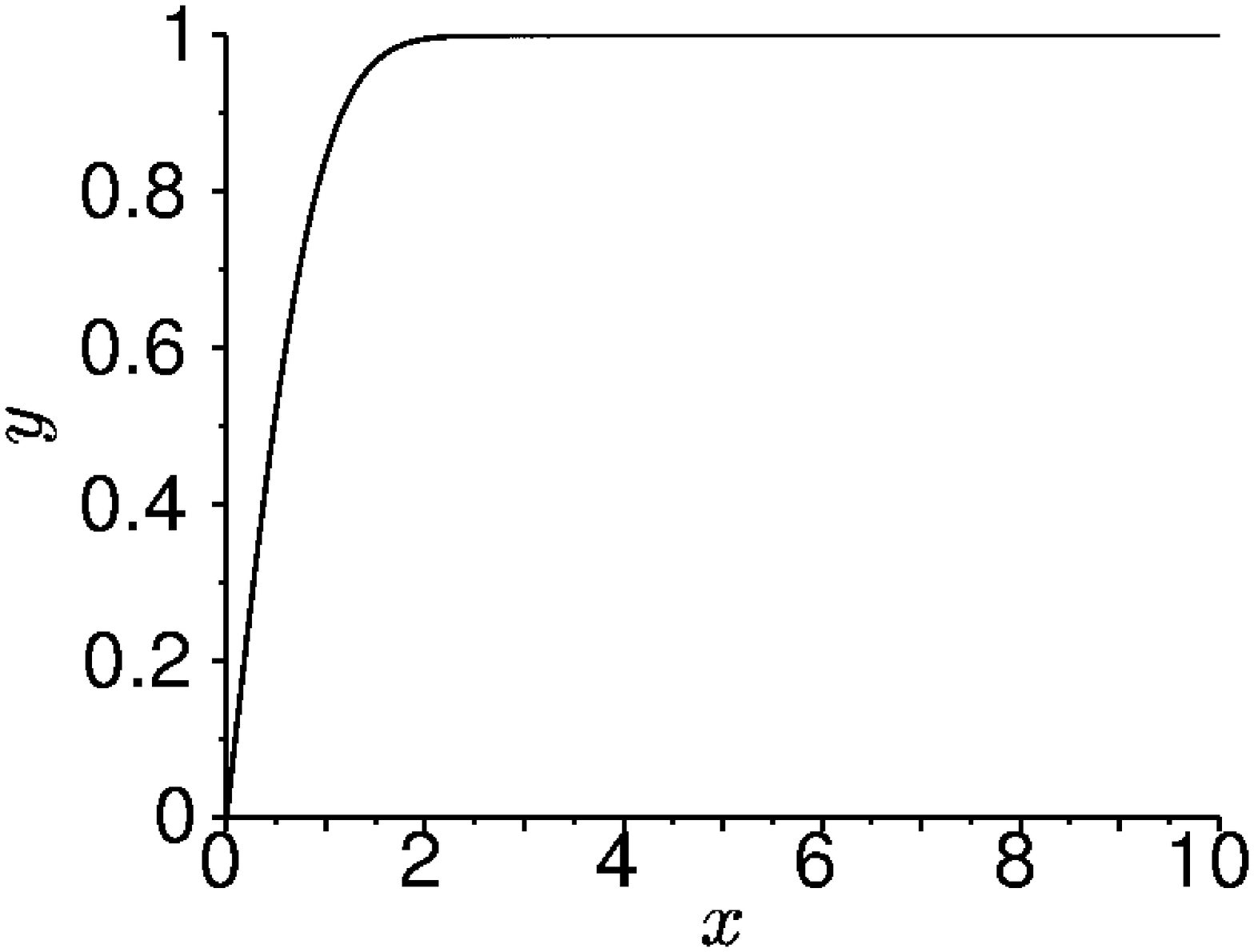}
\caption{$\varphi_0=\erf$}
\end{subfigure}
\begin{subfigure}{0.3\textwidth}
\includegraphics[scale=0.25]{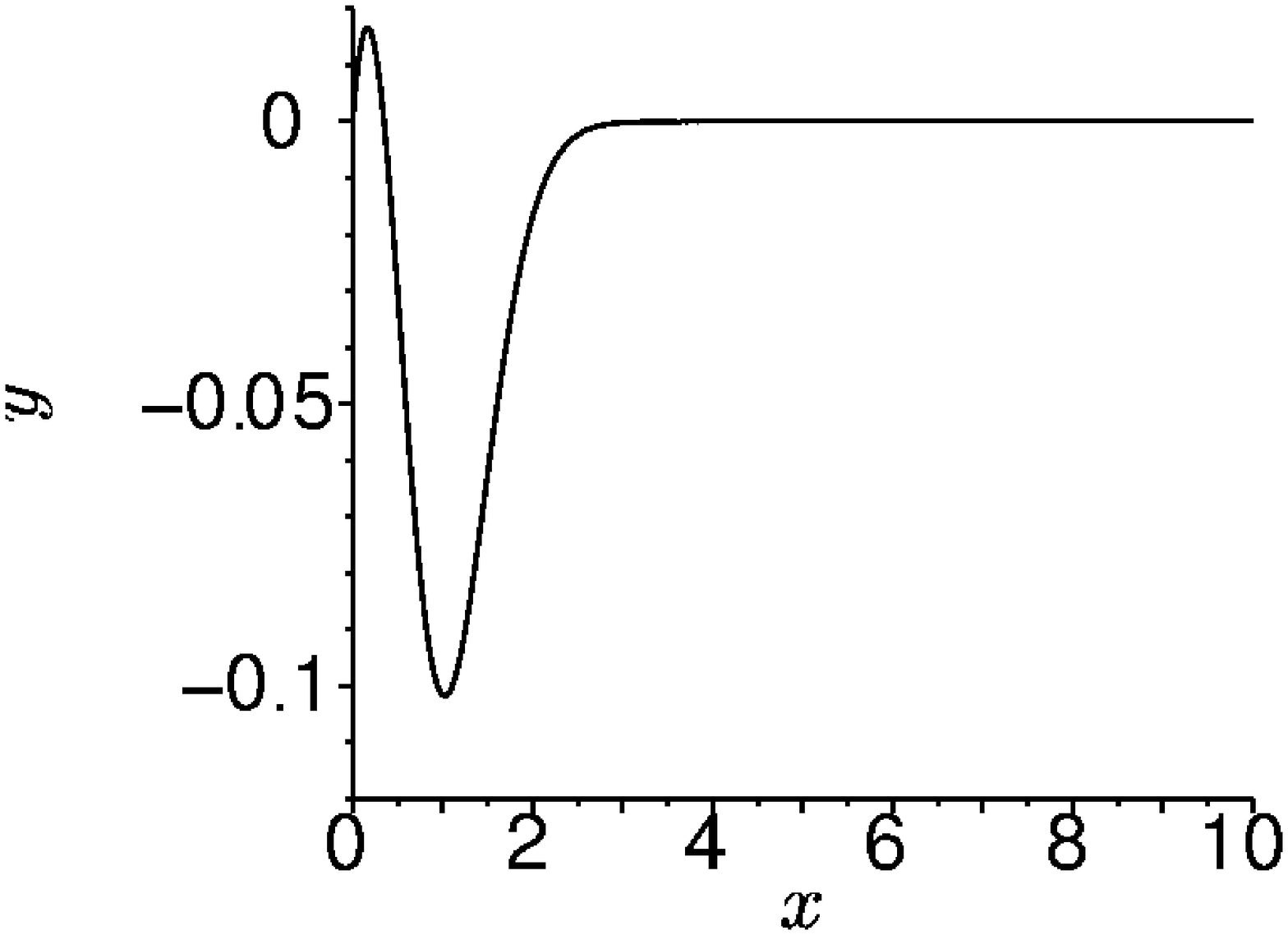}
\caption{$\varphi_1$}
\end{subfigure}
\begin{subfigure}{0.3\textwidth}
\includegraphics[scale=0.25]{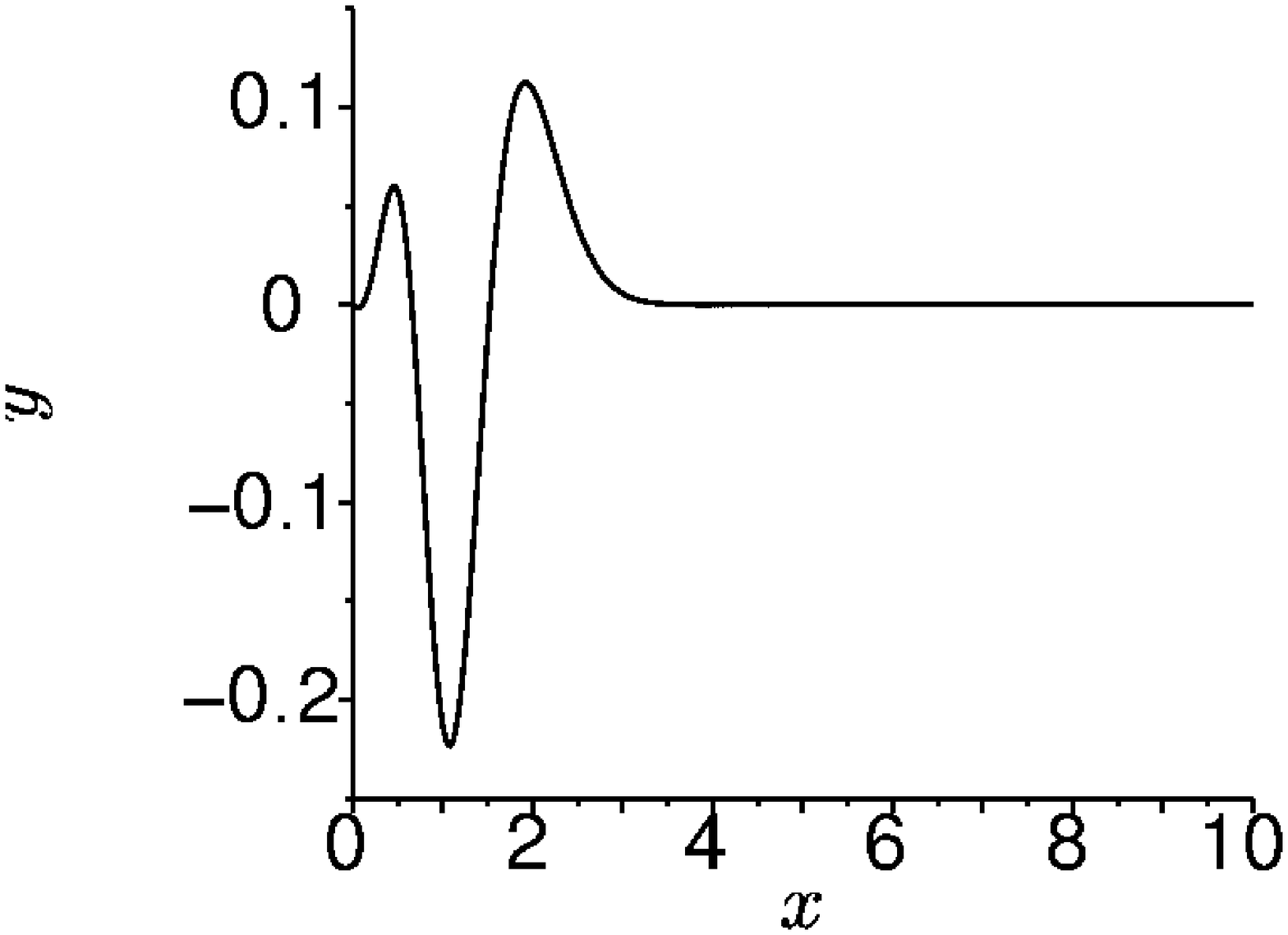}
\caption{$\varphi_2$}
\end{subfigure}
\end{figure}

We will now investigate the relation between the approximations $\Psi_{\delta,1}$, $\Psi_{\delta,2}$, and the modified error function $\Phi_\delta$. The analysis will be again limited to those values of $\delta$ for which it is known the existence and uniqueness of $\Phi_\delta$, i.e. to $0<\delta<\delta_0\simeq 0.2037$ \cite{CeSaTa2017}. In contrast to the analysis presented in Section \ref{Sect:Zero}, investigations here will be based on numerical computations. 

Numerical values for $\Phi_\delta$ were obtained by solving problem (\ref{Pb:y}) through the routine \texttt{bvodeS} implemented in Scilab. The problem was solved for the domain $[0,10]$, by considering a uniform mesh $\mathcal{P}$ with step size $10^{-2}$. 

Let $\mathcal{E}_{\delta,m}$ be the {\em discrete error} between $\Phi_\delta$ and $\Psi_{\delta,m}$, defined by
\begin{equation}
\mathcal{E}_{\delta,m}=\max\{\left|\Psi_{\delta,m}(x)-\Phi_\delta(x)\right|\,:\,
x\in\mathcal{P}\}\quad \text{for }m=0,1,2.
\end{equation}  

Figure \ref{Fig:Comparison}a shows some plots of $\mathcal{E}_{\delta,m}$ for $m=0,1,2$ and $\delta\in[0,0.2]\subset[0,\delta_0)$. On one hand, we find that $\Psi_{\delta,1}$ and $\Psi_{\delta,2}$ are better approximations of $\Phi_\delta$ than $ \Psi_{\delta,0}$. On the other hand, we also find that $\Psi_{\delta,1}$ is better than $\Psi_{\delta,2}$. This, together with the fact that $\Psi_{\delta,1}$ admits an explicit representation in terms of error and exponential functions only (in contrast to $\Psi_{\delta,2}$, which involves some integrals that can not be explicitly computed), turns $\Psi_{\delta,1}$ the best approximation among those proposed in this article. Figure \ref{Fig:Comparison}b shows the comparison between the 
$\Psi_{\delta,1}$ and the modified error function $\Phi_\delta$, for $\delta=0.2$. Though we are not able to find a complete explanation of why $\Psi_{\delta,1}$ approximates better $\Phi_{\delta}$ than $\Psi_{\delta,2}$, we suggest that the numerical implementation of the integrals in the definition of $\varphi_2$ might be introducing non-negligible perturbations.   
 
\begin{figure}[H]  
\caption{{\em Comparisons between the modified error function $\Phi_{\delta}$ 
and its approximations $\Psi_{\delta,m}$ for $m=0,1,2$.}}
\label{Fig:Comparison}
\centering
\begin{subfigure}{0.4\textwidth}
\hspace*{-7.5mm}\includegraphics[scale=0.25]{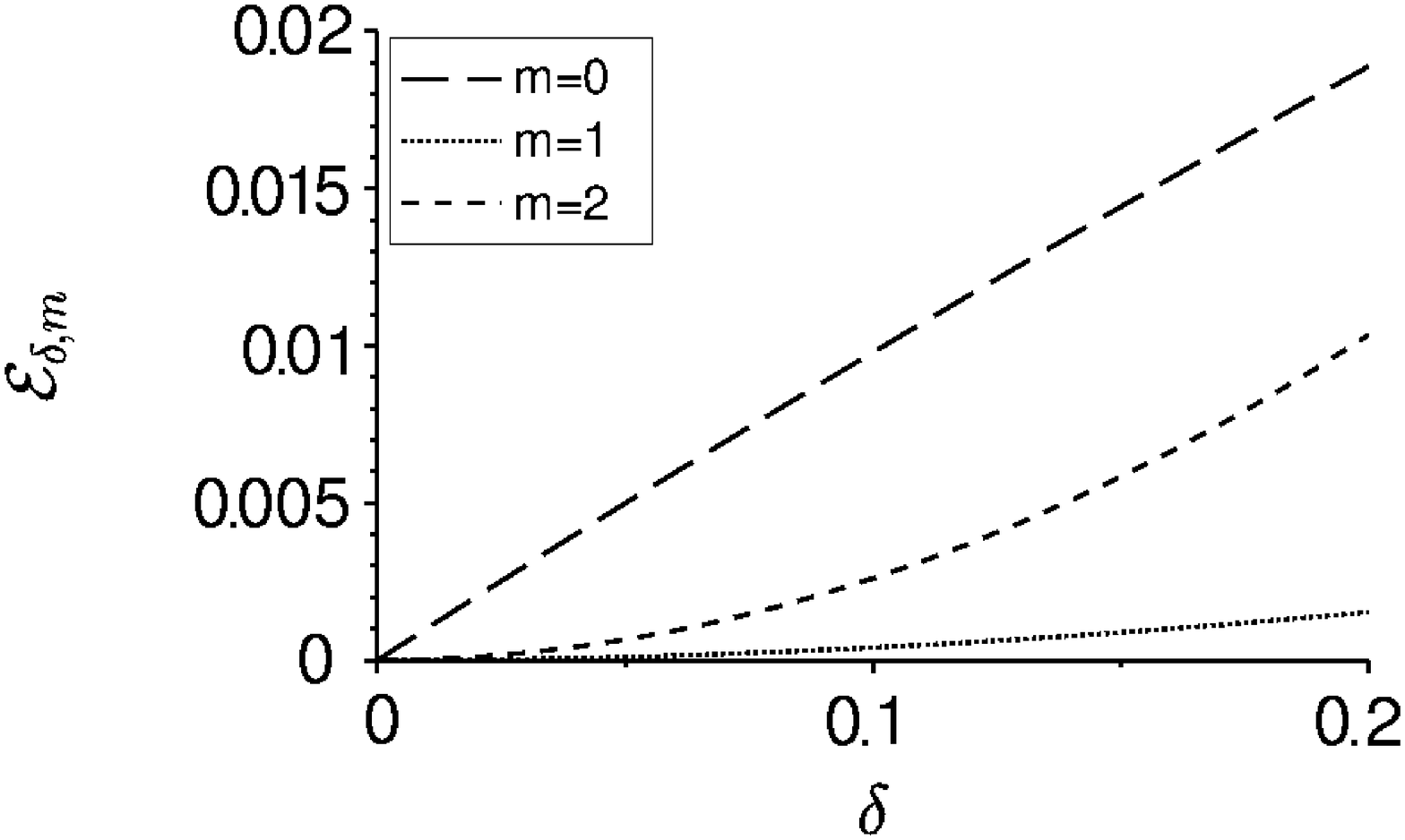}
\caption{ Error $\mathcal{E}_{\delta,m}$ for $m=0,1,2$ and 
$\delta\in[0,0.2]$.}
%\label{Fig:Error}
\end{subfigure}
\hspace{2cm}
\begin{subfigure}{0.42\textwidth}
\includegraphics[scale=0.25]{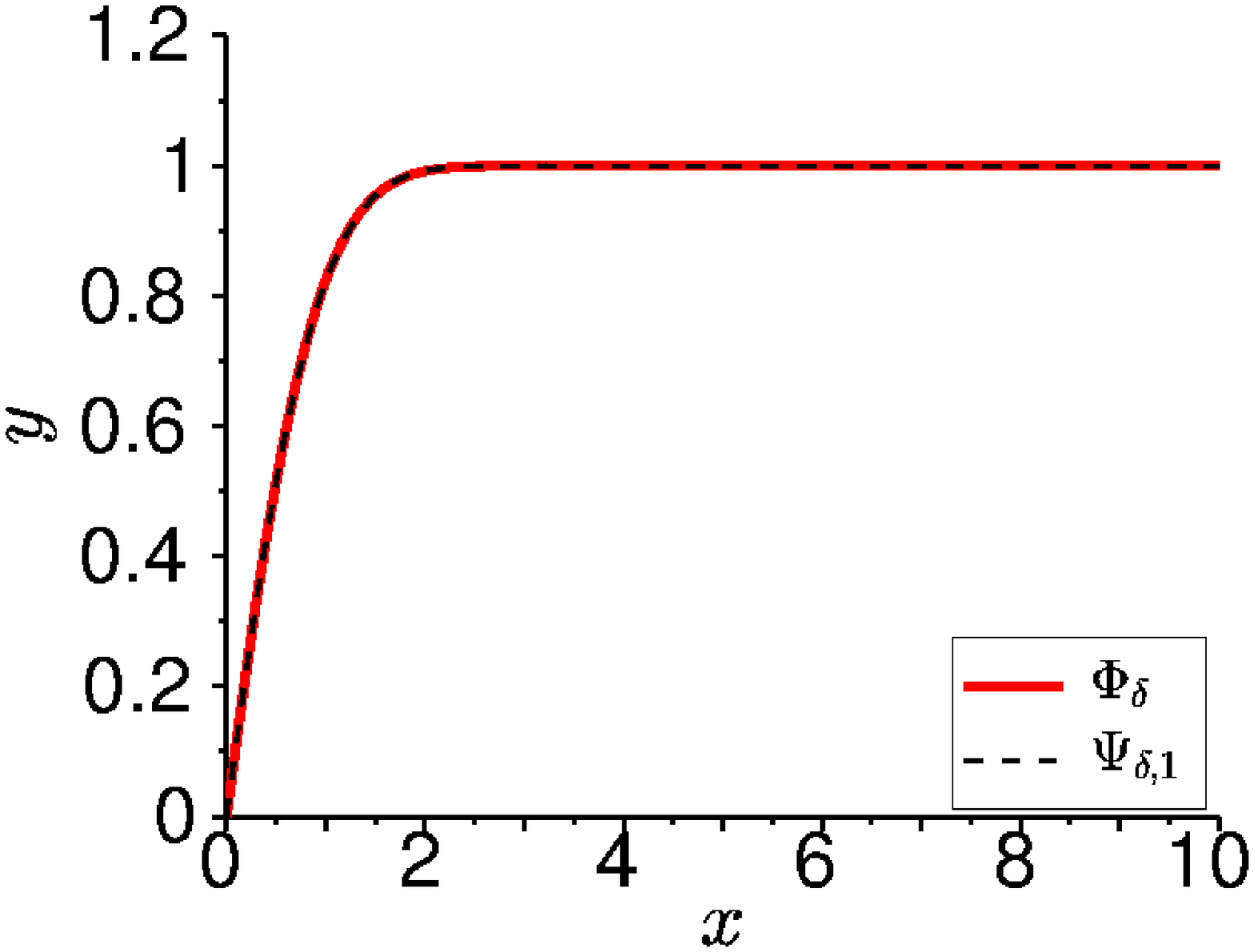}
\caption{ Modified error function $\Phi_\delta$ and its approximation of 
first order $\Psi_{\delta,1}$ for $\delta=0.2$.}
\end{subfigure}
\end{figure}

Finally, in Figure \ref{Fig:ComparisonBis} we present plots for the modified error function $\Phi_\delta$ and the approximation of first order $\Psi_{\delta,1}$ for different values of $\delta>-1$. Even when they do not belong to the interval $[0,\delta_0)$  over which theoretical results on existence and uniqueness of $\Phi_\delta$ are known, very good agreement is obtained. This enforces the former conclusions of being $\Psi_{\delta,1}$ the best approximation of $\Phi_\delta$, among $\Psi_{\delta,m}$ for $m=0,1,2$.

\begin{figure}[H]   
\caption{{\em Modified error function $\Phi_\delta$ and its approximation of 
first order $\Psi_{\delta,1}$ for  $\delta=-0.9,$ $-0.5$, $0.5, 1, 1.5, 
2$.}}
\label{Fig:ComparisonBis}
\centering
\begin{subfigure}{0.4\textwidth}
\includegraphics[scale=0.25]{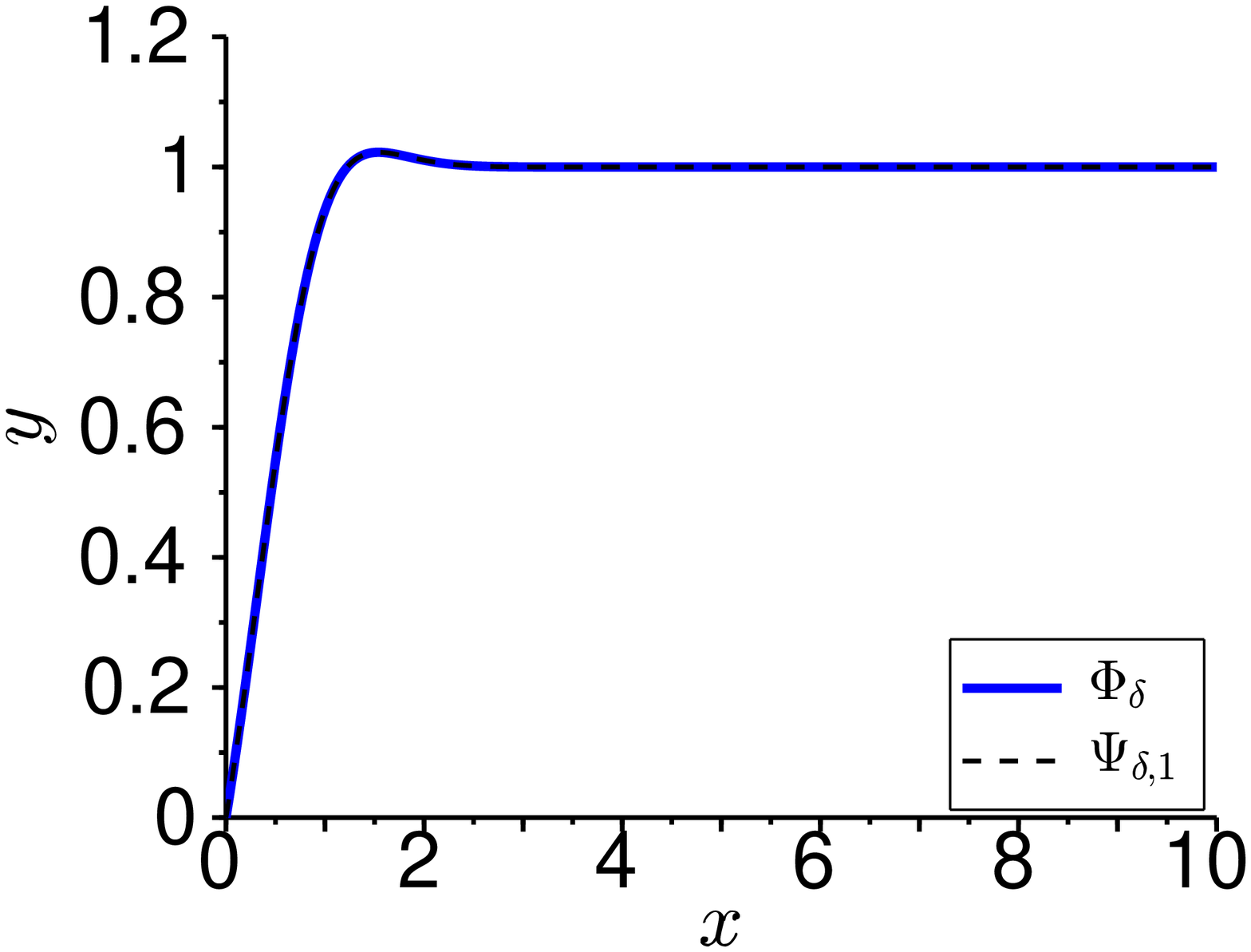}
\caption{ $\delta=-0.9$}
%\label{Fig:Comp09Neg}
\end{subfigure}
\hspace{2cm}
\begin{subfigure}{0.4\textwidth}
\includegraphics[scale=0.25]{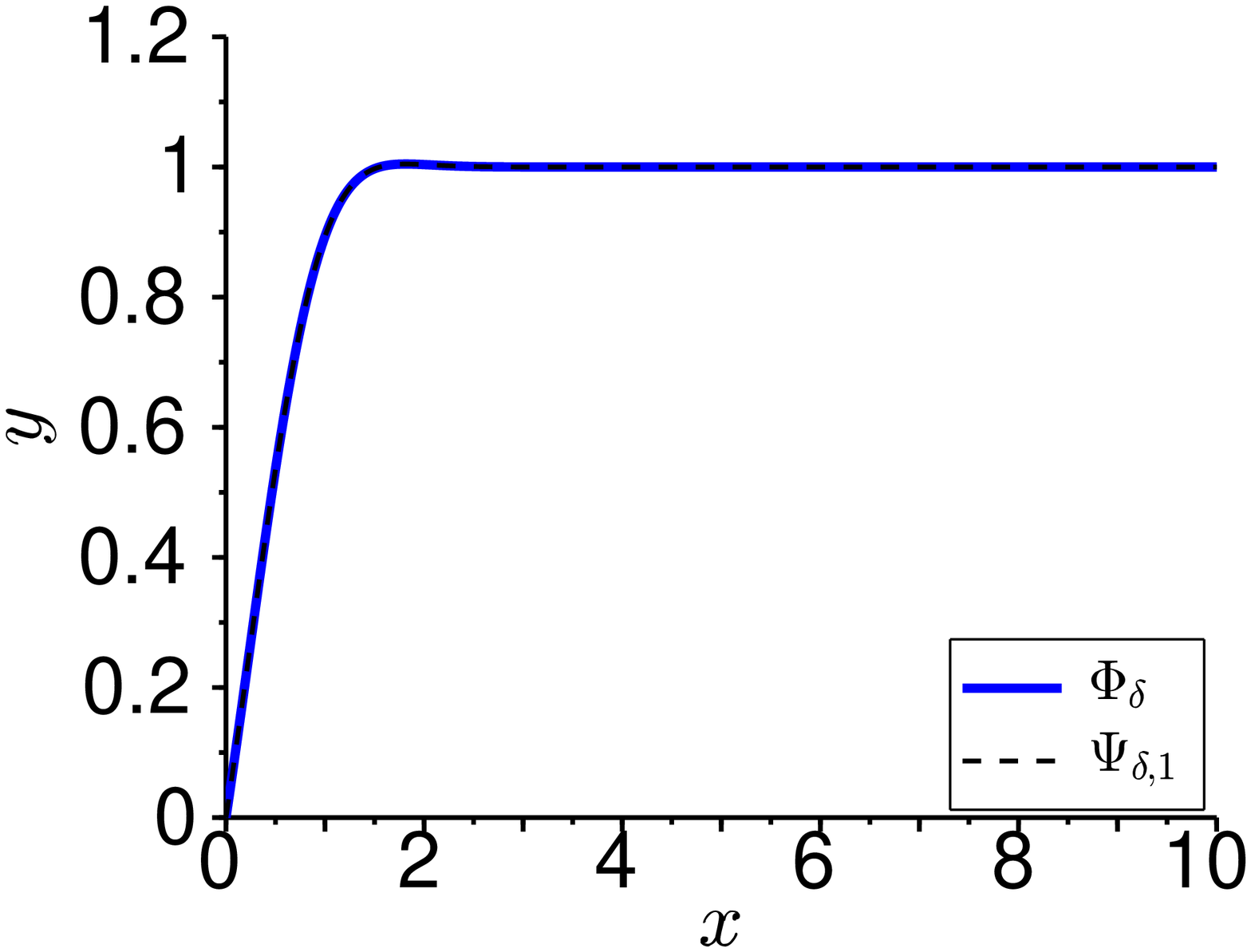}
\caption{ $\delta=-0.5$}
\end{subfigure}\\
\begin{subfigure}{0.4\textwidth}
\includegraphics[scale=0.25]{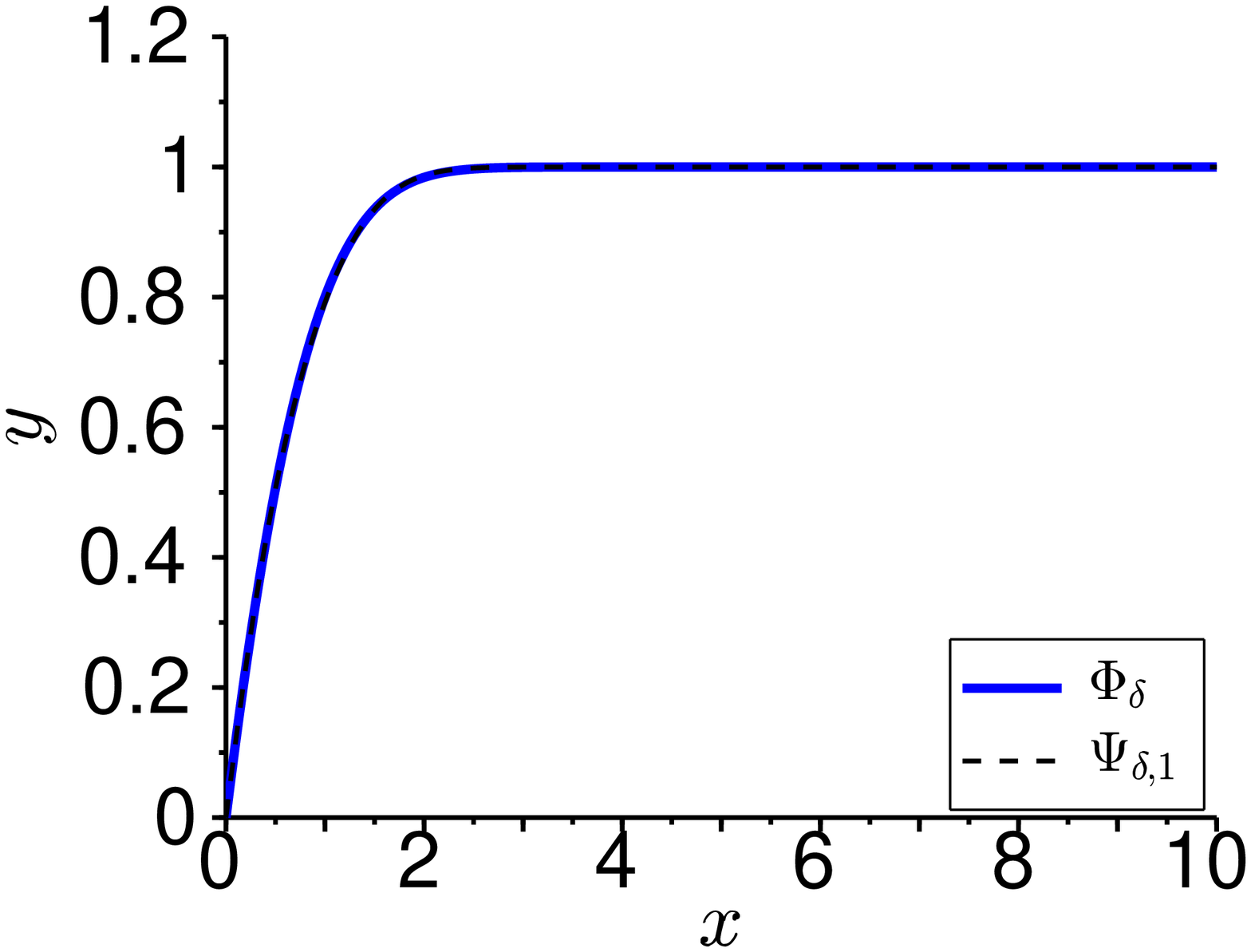}
\caption{ $\delta=0.5$}
\end{subfigure}
\hspace{2cm}
\begin{subfigure}{0.4\textwidth}
\includegraphics[scale=0.25]{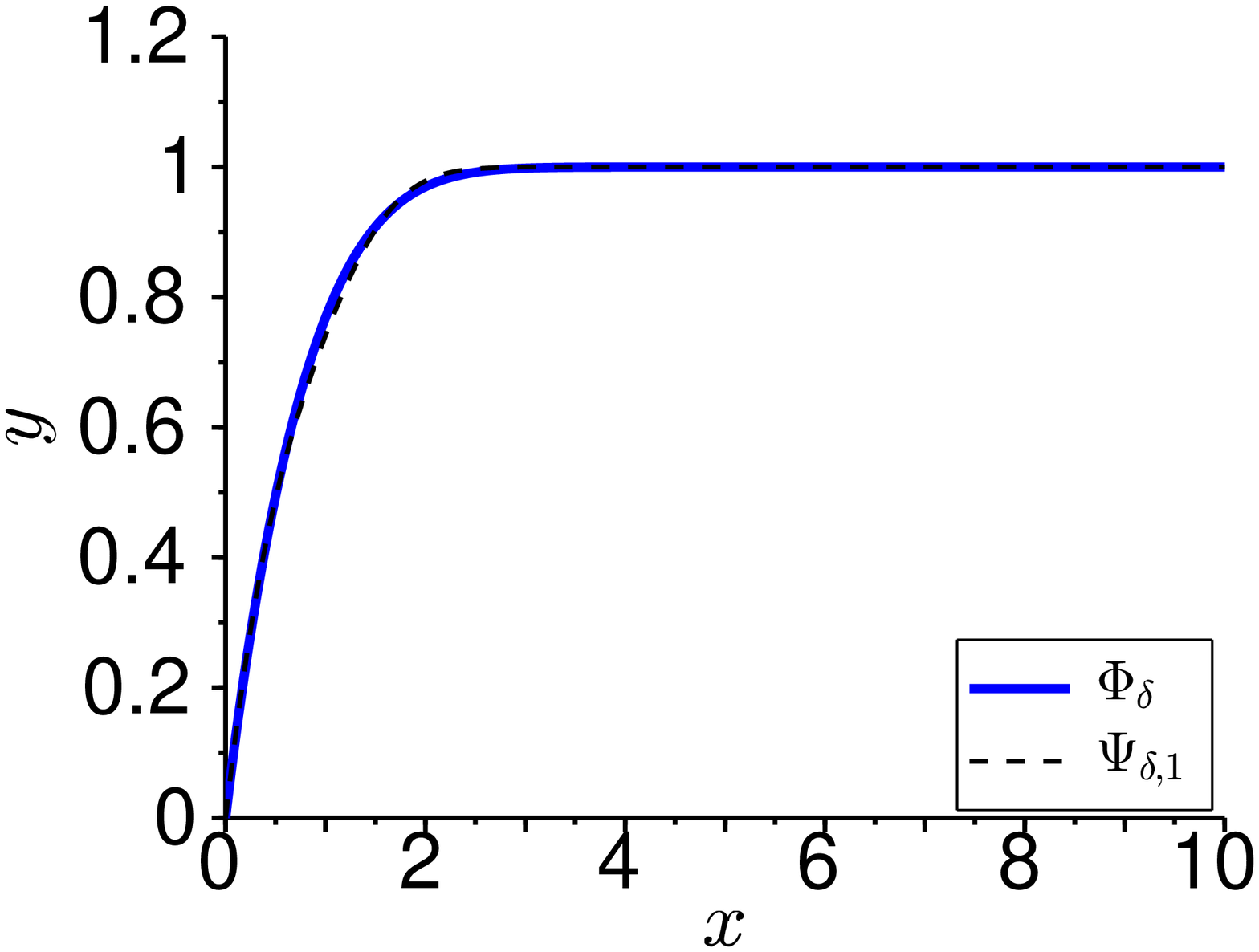}
\caption{ $\delta=1$}
\end{subfigure}\\
\begin{subfigure}{0.4\textwidth}
\includegraphics[scale=0.25]{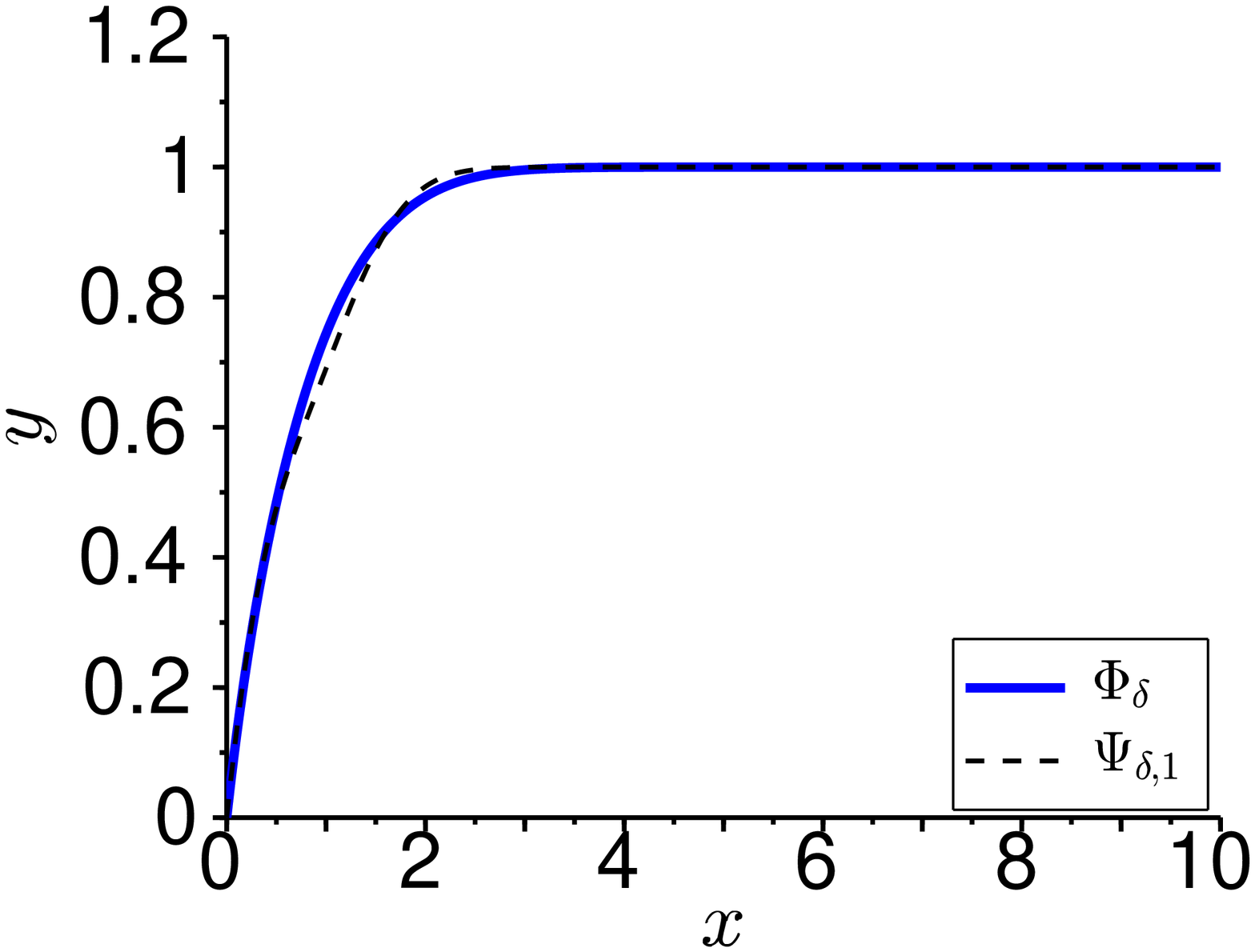}
\caption{ $\delta=1.5$}
\end{subfigure}
\hspace{2cm}
\begin{subfigure}{0.4\textwidth}
\includegraphics[scale=0.25]{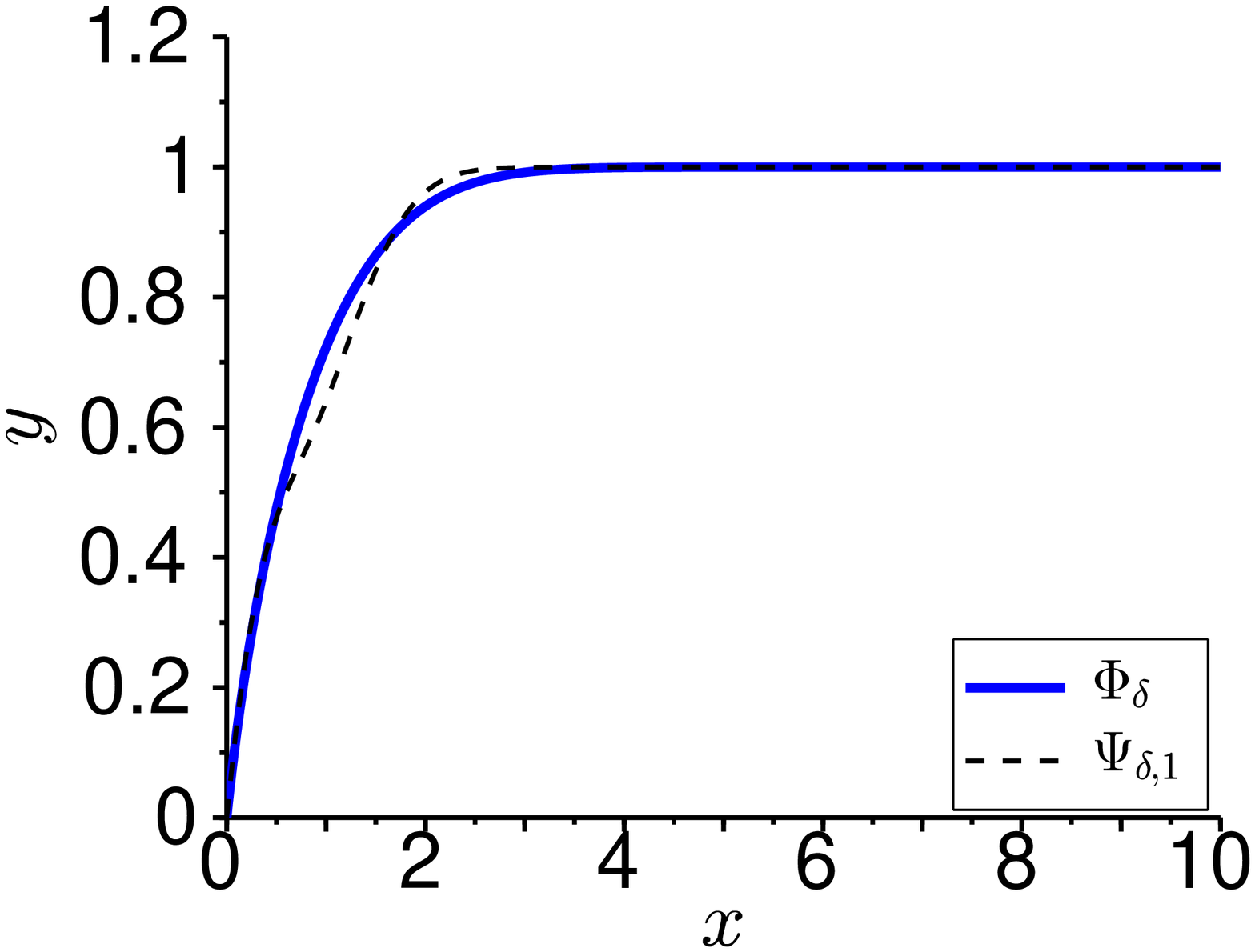}
\caption{ $\delta=2$}
\end{subfigure}

\end{figure}

\section{Properties of the modified error function}\label{Sect:Prop}
We end this article by proving that the modified error function $\Phi_\delta$ found in \cite{CeSaTa2017} shares some basic properties with the classical error function $\erf$. More precisely, those of being an increasing concave non-negative and bounded function.

\begin{theorem}\label{Th:Prop}
If $0<\delta<\delta_0$, then the only solution $\Phi_\delta$ in $K$ to problem (\ref{Pb:y}) satisfies the following properties:
\begin{equation}\label{Prop:Phi}
0 \leq \Phi_\delta(x) \leq 1, 
\quad\quad
\Phi_\delta'(x)>0,
\quad\quad
\Phi_\delta''(x)<0\quad\quad\forall\,x>0.
\end{equation}
\end{theorem}

\begin{proof}
The first property in (\ref{Prop:Phi}) is a direct consequence of the fact that $\Phi_\delta$ belongs to $K$. In order to prove the second one, we start by showing that $\Phi_\delta'(x)\neq 0$ for all $x>0$. We will assume that there exists $x_0>0$ such that $\Phi_\delta'(x_0)=0$ and we will reach a contradiction. Since equation (\ref{eq:y}) can be written as
\begin{equation}\label{Eq:Varphi-bis}
(1+\delta y(x))y''(x)+\delta(y'(x))^2+2x y'(x)=0\quad x>0
\end{equation}
and we know that
\begin{equation}\label{IneqVarphi}
1+\delta \Phi_\delta(x_0)>0,
\end{equation}
we find that $\Phi_\delta''(x_0)=0$. From this, by differentiating (\ref{Eq:Varphi-bis}) and taking (\ref{IneqVarphi}) into consideration, it follows that $\Phi_\delta'''(x_0)=0$. We continue in this fashion obtaining that $\Phi_\delta^{(n)}(x_0)=0$ for all $n\in\mathbb{N}$. This implies that $\Phi_\delta\equiv 0$ in $\mathbb{R}_0^+$, since $\Phi_\delta$ is an analytic function. But the latter contradicts that $\Phi_\delta(+\infty)=1$. Therefore, $\Phi_\delta'(x)\neq 0$ for all $x>0$. This implies that the function $\Phi_\delta'$ does not change its sign in $\mathbb{R}_0^+$. Since $\Phi_\delta(0)\leq 1$ and $\Phi_\delta(+ \infty)=1$, it follows that $\Phi_\delta'(x)> 0$ for all $x>0$. Finally, the last property in (\ref{Prop:Phi}) follows straightforward from the previous ones and the fact that $\Phi_\delta''$ is given by
\begin{equation}
\Phi_\delta''(x)=-\frac{\delta(\Phi_\delta'(x))^2+2x\Phi_\delta'(x)}{1+\delta\Phi_\delta(x)}\quad x>0.
\end{equation}
\end{proof}

\begin{remark}
Note that Figure \ref{Fig:ComparisonBis}a suggest that the assumption 
$\delta>0$ can not be removed from the statement of Theorem \ref{Th:Prop} since 
the two last properties in (\ref{Prop:Phi}) seem to fail. 
\end{remark}

\section{Conclusions}
In this article, we have proposed a method to obtain approximations of the modified error function $\Phi_\delta$ introduced by Cho and Sunderland in 1974 \cite{ChSu1974} as part of a Stefan problem with variable thermal conductivity. This is defined as the solution to a nonlinear boundary value problem for a second order ordinary differential equation which depends on a parameter $\delta>-1$. By assuming that $\Phi_\delta$ admits a power series representation in $\delta$, we proposed some approximations $\Psi_{\delta,m}$ given as the partial sums of the first $m$ terms. It was presented three of them: the zeroth order approximation $\Psi_{\delta,0}=\erf$; the first order approximation $\Psi_{\delta,1}$, which can be written in terms of error and exponential functions only; and the second order approximation $\Psi_{\delta,2}$, which can be written in terms of the error and exponential functions, and some integrals of combinations of them. Analysis of errors between the approximations and the original function was performed by considering only those values of $\delta$ for which existence and uniqueness of the modified error function is known, i.e. to small positive values of the parameter $\delta$. When $m=0$ it was found that $\Phi_{\delta}$ uniformly converges to $\Psi_{\delta,0}$. This suggest that $\Phi_\delta\simeq\Psi_{\delta,0}$ for small values of $\delta$. When $m=1,2$, numerical investigations suggest that $\Psi_{\delta,1}$ and $\Psi_{\delta,2}$ are also accurate approximations for $\Phi_\delta$. In particular, it was obtained that $\Psi_{\delta,1}$ and $\Psi_{\delta,2}$ are better approximations than $\Psi_{\delta,0}$, and that $\Psi_{\delta,1}$ is better than $\Psi_{\delta,2}$. This, together with the simple expression of $\Psi_{\delta,1}$ in terms of the error and exponential functions only, turns the first order approximation the best one among those presented here. The fact that $\Psi_{\delta,1}$ seems to be a better approximation than $\Psi_{\delta,2}$ can not be completely addressed by the authors. Nevertheless, we suggest that the numerical implementation of the integrals in the definition of the second order approximation might be introducing non-negligible perturbations. Comparisons between $\Psi_{\delta,1}$ and $\Phi_{\delta}$ were also presented for values of $\delta>-1$. Good agreement was obtained, even for those values of $\delta$ which do not belong to the theoretical interval determined by the existence and uniqueness results of $\Phi_\delta$. Finally, we proved that the modified error function is an increasing non-negative concave function which is bounded by 1, as the classical error function is, provided $\delta$ assumes small positive values. Results presented here can be used to obtain explicit approximate solutions to Stefan problems for phase-change processes with linearly temperature-dependent thermal conductivity. This investigation suggest that the proof of existence and uniqueness of $\Phi_\delta$ for $-1< \delta<0$ and evaluation of integrals involving exp-, erf-, erfc-functions are still an open problem in the mathematical analysis of Stefan problems.

\section*{Acknowledgements}
This paper has been partially sponsored by the Project PIP No. 0275 from CONICET-UA (Rosario, Argentina) and AFOSR-SOARD Grant FA 9550-14-1-0122. The authors would like to thanks the anonymous referees whose insightful comments have benefited the presentation of this article.

%\section{Some graphics for the modified error functions}
%-------------------------------------------------------------------------------------
%\nocite{*}
\bibliographystyle{plain}
\bibliography{References}
%\bibliography{Ceretani-Salva-Tarzia-References}
%-------------------------------------------------------------------------------------
\end{document}